\documentclass[11pt]{amsart}
\usepackage{amsfonts,amssymb,amsmath,amsthm}
\usepackage{url}
\usepackage{enumerate}
\usepackage[all]{xy}
\usepackage[utf8]{inputenc}
\usepackage{amsmath,amsfonts,amsthm,url,color,amssymb}

\usepackage{graphicx}
\usepackage{algpseudocode, algorithm}
\usepackage{hyperref}
\usepackage{todonotes}

\usepackage[norefs,nocites]{refcheck}

\newcommand{\F}{\mathbb{F}}

\newcommand{\ord}{\mathrm{ord}}

\newtheorem{theorem}{Theorem}[section]
\newtheorem{lemma}[theorem]{Lemma}

\newtheorem{corollary}[theorem]{Corollary}
\newtheorem{proposition}[theorem]{Proposition}
\theoremstyle{definition}
\newtheorem{definition}[theorem]{Definition}
\newtheorem{example}[theorem]{Example}

\theoremstyle{remark}
\usepackage{amssymb}
\newtheorem{remark}[theorem]{Remark}
\usepackage{enumerate}
\author[L. Reis]{Lucas Reis}
\address{Departamento de Matem\'{a}tica, Universidade Federal de Minas Gerais, Belo Horizonte MG, 31270901, Brazil}
\email{lucasreismat@gmail.com}

\thanks{This work was conducted during a pleasant visit of the author at the Mathematics Department of Carleton University - Canada, supported by Coordenação de Aperfeiçoamento de Pessoal de Nível Superior - Brasil (CAPES) -
Finance Code 001. We would like to thank the reviewer for her/his thoughtful comments and efforts towards improving the manuscript.}

\begin{document}

\title{Iterating additive polynomials over finite fields}
\begin{abstract}
Let $q$ be a power of a prime $p$, let $\F_q$ be the finite field with $q$ elements and, for each nonconstant polynomial $F\in \F_{q}[X]$ and each integer $n\ge 1$, let $s_F(n)$ be the degree of the splitting field (over $\F_q$) of the iterated polynomial $F^{(n)}(X)$. In 1999, Odoni proved that $s_A(n)$ grows linearly with respect to $n$ if $A\in \F_q[X]$ is an additive polynomial not of the form $aX^{p^h}$; moreover, if $q=p$ and $B(X)=X^p-X$, he obtained the formula $s_{B}(n)=p^{\lceil \log_p n\rceil}$. In this paper we note that $s_F(n)$ grows at least linearly unless $F\in \F_q[X]$ has an exceptional form and we obtain a stronger form of Odoni's result, extending it to affine polynomials. In particular, we prove that if $A$ is additive, then $s_A(n)$ resembles the step function $p^{\lceil \log_p n\rceil}$ and we indeed have the identity $s_A(n)=\alpha p^{\lceil \log_p \beta n\rceil}$ for some $\alpha, \beta\in \mathbb Q$, unless $A$ presents a special irregularity of dynamical flavour. As applications of our main result, we obtain statistics for periodic points of linear maps over $\F_{q^i}$ as $i\to +\infty$ and for the factorization of iterates of affine polynomials over finite fields. 
\end{abstract}

\maketitle

\section{Introduction}
Given a field $K$, a polynomial $F\in K[X]$ and a positive integer $n$, let $F^{(n)}(X)$ be the $n$-th fold composition of $F$ with itself. Some past works have dealt with polynomials $F$ in which the iterates $F^{(n)}(X)$ present special factorization over $K$: see~\cite{A05,AM00,K85}. The main case of study is when all the iterates $F^{(n)}(X), n\ge 1$ are irreducible over $K$: such polynomials are called {\em stable}. There are a few constructions of stable polynomials~\cite{bin, od}, where the field of coefficients $K$ is a number field or a function field. 
In~\cite{od}, Odoni proves the stability of ``generic" additive polynomials with coefficients in a function field and compute the Galois group of the iterates of such polynomials. Recall that a polynomial $f\in K[X]$ is additive if the identity $f(X+Y)=f(X)+f(Y)$ holds in the polynomial ring $K[X, Y]$. If $K$ has positive characteristic $p$, the additive polynomials in $K[X]$ are precisely those of the form $\sum_{i=0}^ha_iX^{p^i}$ with $a_i\in K$. However, in the case where $K$ is finite, such polynomials are usually not stable. In fact, no {\em affine polynomial} (i.e., $A(X)+b$ with $A(X)$ additive and $b\in K$) of degree at least $2$ can be stable if $K$ is finite: see Theorem~10 in~\cite{sh} for more details.

The stability of polynomials over finite fields have been explored primarily in the case of quadratic polynomials. A relaxation of stable polynomials is considered in~\cite{JB12}, yielding the so called {\em settled polynomials}. A more general problem on the factorization of $F^{(n)}(X)$ is proposed in~\cite{arxiv} in the case where $K$ is finite. In particular, the problem suggests to explore some arithmetic functions associated to $F^{(n)}(X)$ such as the degrees of the irreducible factors (largest, average) and the number of irreducible factors. For more details, see~ Question 18.9 in~\cite{arxiv}. Many issues related to Question 18.9 were eariler discussed in~\cite{GOS} and, more recently, in~\cite{R21}.

Still in the finite field setting, an arithmetic function that can be considered is $s_F(n)$, the smallest positive integer such that $F^{(n)}(X)$ splits completely over $\F_{q^{s_F(n)}}$. For instance, $F$ is stable if and only if $s_F(n)=\deg(F)^n$ for every $n\ge 1$. As previously mentioned, in~\cite{sh} the authors prove that no affine polynomial of degree at least $2$ and with coefficients in a finite field can be stable. In~\cite{odoni}, Odoni proved something much stronger for additive polynomials: if $A\in \F_q[X]$ is an additive polynomial not of the form $aX^{p^h}$, then $s_A(n)$ grows linearly with respect to $n$. More precisely, he proved that the sequence $\{s_A(n)\cdot n^{-1}\}_{n\ge 1}$ is bounded and its set of limit points equals $[\alpha_A, p\alpha_A]$ for some $\alpha_A>0$. For more details, see Theorem~2 in~\cite{odoni}. Odoni's result implies a refined information on the factorization of iterated additive polynomials: the iterates of additive polynomials typically splits into many distinct irreducible factors, all of small degree. He ends the paper with an interesting example, where the constant $\alpha_A$ is explicitly computed: if $q=p$, then $s_{X^p-X}(n)=p^{\lceil \log_p n\rceil}$ for every $n\ge 1$ and, in particular, $\alpha_{X^p-X}=1$. 

Odoni's proof on the growth of $s_A(n)$ with $A$ additive relies on transferring the problem to the skew polynomial ring  $\F_q[X, \sigma]$: this is a natural model for the ring of additive polynomials (with respect to the ordinary sum and composition of polynomials). In this approach, the problem of computing $s_A(n)$ translates into computing valuations of elements in $\F_q[X, \sigma]$ with respect to some ideals of the same ring. However, the latter is quite nontrivial: since the ring $\F_q[X, \sigma]$ is not commutative, we no longer have many useful factorization properties of elements and ideals. In particular, the computations in the paper are involved and technical. Moreover, in his approach he only deals with a subsequence $\{s_A(n\nu)\cdot (n\nu)^{-1}\}_{n\ge 1}$ with fixed $\nu\in \mathbb N$, noticing that such subsequence controls the oscilation of the whole sequence. We actually believe that he was trying to derive a simple closed formula like the one $s_{X^p-X}(n)=p^{\lceil \log_p n\rceil}$ for $q=p$. For more details, see Sections 3, 4 and 5 of~\cite{odoni}.




In this paper, we study the function $s_A(n)$ in more details and recover Odoni's results in a richer form. In particular, we prove that in many cases we have $s_A(n)=\alpha p^{\lceil \log_p \beta n\rceil}$ if $n$ is sufficiently large, where $\alpha, \beta\in \mathbb Q$ do not depend on $n$  and can explictly obtained from some initial values $s_A(1), \ldots, s_A(s_0)$. We also derive a weaker result for affine polynomials $A(X)+b$ and, in particular, we prove that the degree of the splitting fields of the iterates of such polynomials also present linear growth. The latter is a major obstruction in Odoni's approach: there is no identification of affine polynomials in the polynomial ring $\F_q[X, \sigma]$. Before we state our main results, we comment that the linear growth is the slowest possible growth for the the degree of the splitting field of ``generic" iterated polynomials. 

\begin{remark}\label{thm:rmk}
Let $\F_q$ be a finite field of characteristic $p$ and assume that $F\in \F_q[X]$ is not of the form $aX^{s}+b$, where $a\in \F_q^*$ and either $b=0$ or $b\ne 0$ and $s=p^h$ with $h\ge 0$. According to Corollary~2.5 of~\cite{R21}, there exists an absolute constant $c_F>0$ such that, for every positive integer $n$, the polynomial $F^{(n)}(X)$ contains an irreducible factor over $\F_q$ with degree at least $c_F\cdot n$. The latter readily implies that $\frac{s_F(n)}{n}\ge c_F$ for every $n\ge 1$. 
If $F(X)=aX^s+b$ with $a\in \F_q^*$ and either $b=0$ or $b\ne 0$ and $s=p^h$ with $h\ge 0$, it is direct to verify that $s_F(n)=1$ for every $n\ge 1$. In particular, these polynomials are genuine exceptions. 
\end{remark}

Our main result goes as follows.

\begin{theorem}\label{thm:main}
Let $q$ be a power of a prime $p$, let $A(X)\in \F_q[X]$ be an additive polynomial not of the form $aX^{p^h}$ and, for each $b\in \F_q^*$, set $A_b(X)=A(X)+b$. Then the following hold:
\begin{enumerate}[(i)]
\item for every $n\ge 1$, we have $s_A(np)\in \{s_A(n), p\cdot s_A(n)\}$ and, in particular, there exists $0\le i< \log_p n+1$ such that 
\begin{equation}\label{eq:exp}s_A(n)=s_A(1)\cdot p^i.\end{equation}
\item for every $b\in \F_q^*$ and every $n\ge 1$, we have
$s_{A_b}(n)\in \{s_A(n), p\cdot s_{A}(n)\}$;
\item the sequence $\left\{\frac{s_{A}(n)}{n}\right\}_{n\ge 1}$ is bounded and its set of limit points equals $[c_A, pc_A]$ for some $c_A>0$;
\item if $c_A$ is as in the previous item, the sequence $\left\{\frac{s_{A_b}(n)}{n}\right\}_{n\ge 1}$ is bounded and its set of limit points is contained in the interval $[c_A, p^2c_A]$;
\item $c_{A^{(k)}}=k\cdot c_A$ for every $k\ge 1$.
\end{enumerate}
Moreover, for $s_0=\max \{j>0\,:\, s_A(j)=s_A(1)\}$, there exists $A_*\in \F_q[X]$ of the form $R^q$ with $R\in \F_q[x]$ additive and an integer $r\ge 1$ such that 
$$(X^{q^{s_A(1)}}-X)^{q^r}=X^{q^{s_A(1)+r}}-X^{q^r}=A_*(A^{(s_0)}(X)).$$
In particular, $c_A\le \frac{s_A(1)}{s_0}$ with equality if and only if the $\F_p$-linear map $z\mapsto A_*(z)$ is not nilpotent over $\F_{q^{s_A(1)}}$. In the case of equality, we also have
 \begin{equation}\label{eq:main}s_A(n)=s_A(1)p^{\left\lceil \log_p \frac{n}{s_0}\right\rceil},\, n>s_0.\end{equation}
\end{theorem}

Next we provide one numerical example, showing the applicability of Theorem~\ref{thm:main}.

\begin{example}
Let $\alpha\in \F_4$ be such that $\alpha^2=\alpha+1$ and set $A(X)=X^8+\alpha X\in \F_4[X]$. A direct computation yields $s_A(1)=3$ and $s_A(2)=6$, hence $s_0=1$. Moreover, the polynomial $A_*(X)=X^{32}+\alpha^2 X^{4}=(X^{8}+\alpha^2X)^4$ satisfies
$$A_*(A(X))=X^{256}-X^4.$$
It is direct to verify that, in this case, the map $z\mapsto A_*(z)$ is not nilpotent over $\F_{4^3}$. Theorem~\ref{thm:main} implies $c_A=3$ and 
$$s_A(n)=3p^{\lceil \log_pn\rceil}, n\ge 2.$$
\end{example}

\begin{remark}
In the context of Theorem~\ref{thm:main} we see that Eq.~\eqref{eq:main} provides a simple closed formula for $s_A(n)$ whenever the companion polynomial $A_*(X)$ does not induce a nilpotent linear map over the splitting field of $A(X)$. Additive polynomials yielding nilpotent linear maps over extensions of finite fields were previously introduced and explored by the author in~\cite{R18}. However, we discuss other issues there, namely the construction and cycle decomposition of permutation polynomials arising from nilpotent additive polynomials. Nevertheless, such polynomials yield very special dynamical systems over finite fields: if $N\in \F_{q}[X]$ is additive such that $z\mapsto N(z)$ is nilpotent over $\F_{q^M}$, then $0\in \F_{q^M}$ is the only periodic point under this map. In particular, the corresponding directed graph representation of the map $z\mapsto N(z)$ over $\F_{q^M}$ (called {\em functional graph}) contains only one connected component. In the proof of Theorem~\ref{thm:main} we prove that the condition on $A_*(X)$ being nilpotent is also equivalent to the following: for each positive integer $n$, there exists $k=k(n)$ such that $A_*^{(k(n))}(X)$ is divisible by $A^{(n)}(X)$. Since $A, A_*$ are additive, we have that $A_*(0)=A(0)=0$ and so the latter implies that 
$A^{-\infty}(0)\subseteq A_*^{-\infty}(0)$. Here, for a polynomial $F\in \F_q[X]$, the set $$F^{-\infty}(0)=\bigcup_{j\ge 0}\{y\in \overline{\F}_q\,:\, F^{(j)}(y)=0\}$$ is just the {\em reversed orbit} of $0$ under the map $z\mapsto F(z)$ over  $\overline{\F}_q$.
\end{remark}

If we assume that $A(X)\in \F_q[X]$ is of the form $\sum_{i=0}^ma_iX^{q^i}$, we directly obtain an explicit formula for $s_A(n)$ like the one given in Eq.~\eqref{eq:main}: in particular, we can compute the constant $c_A$. The latter generalizes the case  $A(X)=X^p-X$ given by Odoni and, in fact, it covers the class of additive polynomials $A$ with coefficients in $\F_p[X]$ (i.e., we are simply considering the case $q=p$).  For more details, see Theorem~\ref{thm:special}.

As applications of Theorem~\ref{thm:main}, we also provide statistics for the factorization of iterated affine polynomials and for the number of periodic points under maps $z\mapsto A(z)$ with $A$ additive. First, we need the following definition.

\begin{definition}
For each nonconstant polynomial $F\in \F_q[X]$ and each positive integer $n$, let $N_F(n)$ be the number of distinct irreducible factors of $F^{(n)}(X)$ over $\F_q$. Moreover, if $f_{1, n}, \ldots, f_{N_F(n), n}\in \F_q[X]$ are the distinct irreducible factors of $F^{(n)}(X)$ over $\F_q$, we set 
$$\rho_F(n)=\frac{1}{N_F(n)}\sum_{i=1}^{N_F(n)}\deg(f_{i, n}),$$
the average degree of the $f_{i, n}$'s.
\end{definition}

In the following theorem we prove that $\rho_B(n)$ also grows linearly with respect to $n$ if $B$ is a generic affine polynomial.

\begin{theorem}\label{thm:main2}
Let $A\in \F_q[X]$ be an additive polynomial that is not of the form $aX^{p^h}$, fix $b\in \F_q$ and set $B(X)=A(X)+b$. Then there exists $\alpha_B, \beta_B>0$ such that $\frac{\rho_B(n)}{n}\in [\alpha_B, \beta_B]$ for every $n\ge 1$.
\end{theorem}

Proposition 5.18 in~\cite{R21} essentially proves Theorem~\ref{thm:main2} in the case where $A(X)$ is of the form $\sum_{i=0}^ma_iX^{q^i}$. However, this case is much simpler since we can even describe the degree distribution of the irreducible factors of $A^{(n)}(X)$ over $\F_q$. In particular, the techniques employed there cannot be extended to arbitrary additive polynomials. See Subsection 5.2 in~\cite{R21} for more details.

The following theorem entails that the proportion of periodic points under additive polynomial maps is, in general, irregular through the finite extensions of $\F_q$.

\begin{theorem}\label{thm:main3}
Let $A\in \F_{q}[X]$ be an additive polynomial and, for each integer $n$, let $\pi_A(n)$ be the number of elements in $\F_{q^n}$ that are periodic under the map $z\mapsto A(z)$. Then the limit $$\lim\limits_{n\to \infty}\frac{\pi_A(n)}{q^n},$$
does not exist unless $A(X)=aX^{p^h}$ for some $a\in \F_q$ and some $h\ge 0$. In the latter case, the limit equals $0$ if $a=0$ and equals $1$ if $a\ne 0$.
\end{theorem}

Observe that $\frac{\pi_A(n)}{q^n}$ is just the {\em proportion} of points in $\F_{q^n}$ that are periodic with respect to the map $z\mapsto A(z)$. The latter can be considered in a more general context, where $A$ is replaced by a generic polynomial $F\in \F_q[X]$. Very little is known on the behaviour of $\frac{\pi_F(n)}{q^n}$: some related problems are considered in~\cite{derek, Juul}. More recently, in~\cite{bri} the proportion of periodic points is explored in the case of quadratic maps over the projective space $\mathbb P^{1}(\F_{q^n})$, where the analogue limit equals zero (up to some exceptional cases). An earlier work~\cite{mich} explored the proportion of periodic points for the case of monomials and Chebyshev polynomials. In particular, for these two classes of polynomials, the corresponding limit does not exist in general. Along with monomials and Chebyshev polynomials, additive polynomials have {\em wild} dynamical behaviour. For instance, these three families of polynomials cover most of the known {\em exceptional polynomials} (i.e., polynomials whose induced map permute infinitely many finite extensions of $\F_q$). For more details on exceptional polynomials, see Section 8.4 of~\cite{HB}. 

It is worthy mentioning that the dynamics of Chebyshev polynomials and monomials are well understood by means of their {\em functional graphs} (a directed graph representation of the iteration of maps). In fact, the results 
in~\cite{QR19} imply a simple description of the graphs for such polynomials and also for some special additive polynomials.
On the other hand, the dynamics of arbitrary additive polynomials over finite fields is only well understood from the Linear Algebra point of view: see~\cite{bach, elspas, toledo, hua, reis}. Our proof for Theorem~\ref{thm:main3} relies on finding finite extensions of $\F_q$ containing small/large number of roots of the iterates of $A$. However, we do not obtain a complete description of the dynamics of the map $z\mapsto A(z)$. Instead, those extensions are detected by employing item (i) of Theorem~\ref{thm:main}, combined with Theorem~\ref{thm:dyn}, a general basic result on the structure of periodic points of linear maps over finite fields. In particular we prove that, as $n\to \infty$, the quotient $\frac{\pi_A(n)}{q^n}$ must oscilate between zero and a positive constant unless $A(X)=aX^{p^h}$ for some $a\in \F_q$ and some $h\ge 0$.


In the proof of our results we employ many arithmetic properties of additive polynomials, regarding them as elements of the polynomial ring $\F_q[X]$ but also as maps over the algebraic closure $\overline{\F}_q$ of $\F_q$. A crucial tool is Lemma~\ref{lem:lin}, where we compile many elementary results on the composition of additive polynomials over finite fields, including some commuting properties.

We end this section with the structure of the paper. In Section 2 we provide basic machinery and in Section 3 we prove our main result, Theorem~\ref{thm:main}. Finally, in Section 4 we provide an enhanced version of Theorem~\ref{thm:main} that applies to a subclass of additive polynomials, along with some applications of the same theorem.

\section{Preparation}
For the rest of this paper, we fix $q$ a power of a prime $p$.  In this short section we provide two technical lemmas that will be frequently employed. 

\begin{lemma}\label{lem:lin}
Let $A, B\in \F_q[X]$ be additive polynomials with $A(X)$ nonzero. Then $A(X)$ divides $B(X)$ if and only if there exists another additive polynomial $C\in \F_q[X]$ such that $B(X)=C(A(X))$. Moreover, if $B(X)=C(A(X))$ and $B(X)=\sum_{i=0}^sa_iX^{q^i}$ for some $a_i\in \F_p$, then the following hold:
\begin{enumerate}[(i)]
\item $A(C(X))=B(X)$, i.e., $A(C(X))=C(A(X))$;
\item if $A(X)=A_0^{(k)}(X)$ for some $k\ge 1$ and some additive $A_0\in \F_q[X]$, then $A_0(C(X))=C(A_0(X))$.
\end{enumerate}
\end{lemma}
\begin{proof}
The first statement is a routine exercise and can be proved, for instance, by induction on the degree of $B$; see also Exercise 3.68  in~\cite{LiNi}. For the proof of item (i), it suffices to prove that $A(X)$ and $B(X)$ commute. In fact, in this case, the equality $B(X)=C(A(X))$ implies that 
$$B(A(X))=A(B(X))=A(C(A(X))),$$
and so by setting $Y=A(X)$ we obtain the polynomial identity $B(Y)=A(C(Y))$. The fact that $A(X)$ and $B(X)$ commute follows by induction on the number of nonzero coefficients in $A(X)$ so we omit details. Item (ii) follows in a similar way, noticing that the polynomials $A_0(X)$ and $B(X)$ commute. 
\end{proof}
Given a polynomial $f(X)=\sum_{i=0}^ra_iX^i\in \F_q[X]$, set $L_f(X)=\sum_{i=0}^ra_iX^{q^i}$. The polynomial $L_f(X)$ is commonly known as the {\em $q$-linearized associate} to $f(X)$. The following lemma is easily verified.

\begin{lemma}\label{lem:ass}
For $f, g\in \F_q[X]$, the following hold:
\begin{enumerate}[(a)]
\item $L_{f+g}(X)=L_f(X)+L_g(X)$;
\item $L_{fg}(X)=L_{f}(L_{g}(X))$. In particular, for every $n\ge 1$, we have $L_{f}^{(n)}(X)=L_{f^n}(X)$. 
\item $L_{c}(X)=cX$ for every $c\in \F_q$.
\end{enumerate}
\end{lemma}

\section{Proof of Theorem~\ref{thm:main}}
In this section we prove our main result. We fix the following important notation: for integers $r, s\ge 0$ with $s>0$, set 
$$\mathcal S_{s, r}(X)=X^{q^{s+r}}-X^{q^r}=(X^{q^s}-X)^{q^r}.$$
Lemma~\ref{lem:ass} implies the following important identity that will be further used: for every integer $i\ge 0$,
\begin{equation}\label{eq:crit}\mathcal S_{s, r}^{(p^i)}(X)=\mathcal S_{sp^i, rp^i}(X).\end{equation}
We have the following auxiliary result.

\begin{lemma}\label{key}
Let $A\in \F_q[X]$ be a nonzero additive polynomial and let $j$ be a positive integer. Then the following are equivalent:

\begin{enumerate}[(i)]
\item $A(X)$ splits completely in $\F_{q^j}$;
\item $A(X)$ divides $\mathcal S_{j, r}(X)$ for some $r\ge 0$;
\item there exists $r\ge 0$ and an additive polynomial $B\in \F_q[X]$ such that $B(A(X))=\mathcal S_{j, r}(X)$. 
\end{enumerate}
Moreover, if $A$ is separable, we can assume that $r=0$. 
\end{lemma}
\begin{proof}
For the implication (i)$\to$(ii), assume that $A(X)$ splits completely in $\F_{q^j}$. Hence every irreducible divisor of $A(X)$ over $\F_q$ divides $X^{q^j}-X$.
Take $r$ large enough so that, in the factorization of $A(X)$ over $\F_{q}$, no irreducible factor has multiplicity greater than $q^r$. Hence $A(X)$ divides $(X^{q^{j}}-X)^{q^r}=\mathcal S_{j, r}(X)$. The implication (ii)$\to$(iii) follows directly by Lemma~\ref{lem:lin}. For the implication (iii)$\to$(i), assume that $B(A(X))=\mathcal S_{j, r}(X)$. Again, from Lemma~\ref{lem:lin}, it follows that $A(X)$ divides $\mathcal S_{j, r}(X)=(X^{q^{j}}-X)^{q^r}$ and so the roots of $A(X)$ are also roots of $X^{q^j}-X$. In other words, $A(X)$ splits completely in $\F_{q^j}$. Finally, if $A(X)$ is separable, then $A(X)$ divides $\mathcal S_{j, r}(X)=(\mathcal S_{j, 0}(X))^{q^r}$ if and only if it divides $S_{j, 0}(X)$.
\end{proof}

We end this part with an important remark that is used along the way without further mention. If $A\in \F_q[X]$ is nonzero and additive, then $A(0)=0$ and so the splitting field of $A^{(n)}(X)$ contains the splitting field of $A^{(j)}(X)$ whenever $1\le j\le n$. In particular, for every $1\le j<n$, the number $s_{A}(j)$ divides $s_A(n)$. 


\subsection{Proof of Theorem~\ref{thm:main}}
We prove the items separately.

\begin{enumerate}[(i)]
\item Fix $n\ge 1$ and let $\F_{q^K}$ be the splitting field of $A^{(n)}(X)$ over $\F_q$, that is, $s_A(n)=K$. From Lemma~\ref{lem:lin}, there exists an additive polynomial $B\in \F_q[X]$ and an integer $r\ge 0$ such that \begin{equation}\label{eq:it} B(A^{(n)}(X))=\mathcal S_{K, r}(X),\end{equation} and  the polynomials $A, B$ commute. In particular, taking the $p$-th iterate on both sides of Eq.~\eqref{eq:it}, we obtain
$$B^{(p)}(A^{(np)}(X))=\mathcal S_{K, r}^{(p)}(X)=\mathcal S_{pK, pr}(X).$$ 
From Lemma~\ref{key}, the polynomial $A^{(np)}(X)$ splits completely in $\F_{q^{pK}}$. Therefore, $s_{A}(np)$ divides $pK$. 
Recall that, as $np>n$, the number $s_{A}(np)$ is divisible by $s_A(n)=K$. In particular, we obtain $$s_A(np)\in \{K, pK\}=\{s_A(n), p\cdot s_A(n)\}.$$
The latter readily implies that $s_{A}(p^j)\le s_A(1)\cdot p^j$ for every $j\ge 0$. Recall that $s_A(m)$ divides $s_A(n)$ if $m<n$. In particular, it is clear that for every $n\ge 1$ we have $s_A(n)=s_A(1)p^i$ for some $i\ge 0$. Moreover, if $n\ge 1$ and $j=\lceil \log_p n\rceil<1+\log_p n$, we have $n\le p^j$ and so $$s_A(n)\le s_A(p^j)\le s_A(1)\cdot p^j.$$

\item It is direct to verify that there exists a sequence $\{\beta_{\ell}\}_{\ell\ge 1}$ of elements in $\F_q$ such that $A_b^{(\ell)}(X)=A^{(\ell)}(X)+\beta_{\ell}$ for every $\ell\ge 1$. Now fix $n\ge 1$ and let $\F_{q^S}$ and $\F_{q^T}$ be the splitting fields of $A_b^{(n)}(X)$ and $A^{(n)}(X)$, respectively. In particular, the roots of $A_b^{(n)}(X)$ comprise an $\F_p$-affine space $\mathcal L\subseteq \F_{q^S}$ whose difference set $\mathcal L-\mathcal L=\{y-z\,:\, y, z\in \mathcal L\}$ coincides with the set of the roots of $A^{(n)}(X)$. Therefore, the roots of $A^{(n)}(X)$ lie in $\F_{q^S}$ and so $T$ divides $S$. Conversely, since $A^{(n)}(X)$ splits completely in $\F_{q^T}$, Lemma~\ref{key} entails that $\mathcal S_{T, r}(X)=B(A^{(n)}(X))$ for some additive polynomial $B\in \F_q[X]$ and some $r\ge 0$. In particular, 
$$\mathcal S_{T, r}(X)\equiv -B(\beta_n)\pmod {A_b^{(n)}(X)}.$$
However, recall that $\beta_n\in \F_q$ and $B\in \F_q[X]$, hence $B(\beta_n)\in \F_q$ and so $\mathcal S_{T, r}(X)$ vanishes at $-B(\beta_n)$. In conclusion, 
$$\mathcal S_{T, r}^{(2)}(X)\equiv \mathcal S_{T, r}(-B(\beta_n))\equiv 0\pmod {A_b^{(n)}(X)}.$$ Since $p\ge 2$, we obtain
$$\mathcal S_{Tp, rp}(X)=\mathcal S_{T, r}^{(p)}(X)\equiv \mathcal S_{T, r}^{(p-2)}(0)\equiv 0\pmod {A_b^{(n)}(X)}.$$
Observe that, for $y\in \overline{\F}_q$, we have $S_{Tp, rp}(y)=0$ if and only if $y\in \F_{q^{Tp}}$. Hence $A_b^{(n)}(X)$ splits completely in $\F_{q^{pT}}$ and so $S$ divides $pT$. Since $T$ divides $S$, we conclude that $S\in \{T, pT\}$. In other words, $$s_{A_b}(n)\in \{s_A(n), p\cdot s_A(n)\}.$$

\item For simplicity, we write $M=s_A(1)$ and $\gamma_{A}(n)=\frac{s_A(n)}{n}$. Item (i) entails that $s_{A}(n)<Mp^{1+\log_p n}=Mpn$ for every $n\ge 1$. In particular, the  sequence $\left\{\gamma_A(n)\right\}_{n\ge 1}$ is bounded. For each integer $i\ge 0$, item (i) and Remark~\ref{thm:rmk} entail that 
$$s_i=\max \{j>0\,:\, s_A(j)=Mp^i\},$$
is a well defined positive integer. Since $ps_i>s_i$, we obtain $s_A(ps_i)\ge Mp^{i+1}$. Item (i) implies $s_A(ps_i)\le Mp^{i+1}$ and so $s_A(ps_i)= Mp^{i+1}$. In particular, it follows by the definition that $s_{i+1}\ge ps_i$. The latter combined with the identity $s_A(s_{i+1})=ps_A(s_i)$ implies that the sequence ${\bf \Gamma}:=\{\gamma_A(s_i)\}_{i\ge 0}$ is non increasing. Since $A(X)$ is not of the form $aX^{p^j}$, Remark~\ref{thm:rmk} entails that ${\bf \Gamma}$ is bounded below by a positive real number. Hence there exists $c_A>0$ such that
$$\lim\limits_{i\to +\infty} \gamma_A(s_i)=c_A.$$
If $n$ is an integer with $s_i<n\le s_{i+1}$, then $s_A(n)=Mp^{i+1}$ and so 
$$\gamma_A(s_{i+1})=\frac{Mp^{i+1}}{s_{i+1}}\le \gamma_A(n)=\frac{Mp^{i+1}}{n}<\frac{Mp^{i+1}}{s_i}=p\gamma_A(s_i).$$
In particular, $$c_A=\liminf\limits_{n\to+\infty}\gamma_A(n)\le \limsup\limits_{n\to +\infty} \gamma_A(n)\le pc_A.$$ The latter implies that the set of limit points of $\{\gamma_A(n)\}_{n\ge 0}$ is contained in the interval $[c_A, pc_A]$. Since such set of limit points must be closed, it suffices to prove that the sequence $\left\{\gamma_A(n)\right\}_{n\ge 1}$  is dense in the open interval $(c_A, pc_A)$. For each real number $\alpha\in (1, p)$ and each integer $i\ge 0$, set $t_i=\left\lceil  \frac{ps_i}{\alpha}\right\rceil$. In particular, $t_i=\frac{ps_i}{\alpha}+\kappa_i$ for some $\kappa_i\in [0, 1)$. It follows by the definition that $s_i< t_i\le ps_i$. Recall that $s_{i+1}\ge ps_i$, hence $s_A(t_i)=Mp^{i+1}=ps_A(s_i)$. Since $\{s_i\}_{i\ge 0}$ is unbounded, the sequence
$$\gamma_A(t_i)=\frac{ps_A(s_i)}{t_i}=\frac{\alpha s_A(s_i)}{s_i+\frac{\alpha}{p}\kappa_i},\, i\ge 0\,, $$
converges to $\alpha c_A$. 

\item This item follows by item (iii) and the following restatement of item (ii): $$\frac{s_{A_b}(n)}{n}\in \left\{\frac{s_A(n)}{n}, \frac{ps_A(n)}{n}\right\},\, n\ge 1.$$

\item Fix $k$ a positive integer. We employ the notation of the previous items. Observe that $s_{A^{(k)}}(n)=s_A(nk)$ for every $n\ge 1$ and so $\gamma_{A^{(k)}}(n)=k\gamma_A(n)$. From this fact and item (iii), we conclude that the set of limit points of $\{\gamma_{{A}^{(k)}}(n)\}_{n\ge 1}$ is contained in the interval $[kc_A, pkc_A]$. Hence $c_{A^{(k)}}\ge kc_A$ and $pc_{A^{(k)}}\le pkc_A$. In conclusion, $c_{A^{(k)}}=kc_A$. 
\end{enumerate}
We proceed to the proof of the remaining statements in Theorem~\ref{thm:main}. From hypothesis, the polynomial $A^{(s_0)}(X)$ splits completely in $\F_{q^M}$, hence  Lemma~\ref{key} entails that there exists $r\ge 0 $ and $A_*\in \F_q[X]$ such that $$A_*(A^{(s_0)}(X))=X^{q^{s_A(1)+r}}-X^{q^r}.$$ Replacing $r$ by $r+1$ and $A_*$ by $A_*^q$ if necessary, we can always assume that $r\ge 1$ and that the polynomial $A_*$ is of the form $R^q$ with $R\in \F_q[x]$ additive. As before, for each integer $i\ge 0$, set $s_i=\max \{j>0\,:\, s_A(j)=Mp^i\}$, where $M=s_A(1)$. From Lemma~\ref{lem:lin}, the polynomials $A$ and $A_*$ commute. In particular, as in the proof of item (i), we obtain
$$A_*^{(p^i)}(A^{(s_0p^i)}(X))=X^{q^{Mp^{i}+rp^i}}-X^{q^{rp^i}}=\mathcal S_{Mp^i, rp^i}(X).$$
Hence $s_i\ge s_0p^i$ with equality if and only if $A_*^{(p^i)}(X)$ is not of the form $E(A(X))$ for some additive $E\in \F_q[X]$. In particular, $\frac{s_A(s_i)}{s_i}\le \frac{M}{s_0}$ with equality if and only if $s_i=s_0p^i$. Recall that the sequence $\left\{\frac{s_A(s_i)}{s_i}\right\}_{i\ge 0}$ is non increasing and, by the definition, it converges to $c_A$. 
Hence $c_A\le \frac{M}{s_0}$ with equality if and only if $\frac{s_A(s_i)}{s_i}=\frac{M}{s_0}$ for every $i\ge 0$, that is, $s_i=s_0p^i$ for every $i\ge 0$. As we have seen, the latter is equivalent to $A_*^{(p^{i})}(X)$ not being of the form $E(A(X))$ with $E\in \F_q[X]$ additive, for every $i\ge 0$. Since $A_*$ is of the form $R^q$ with $R\in \F_q[x]$ additive, we see that $A_*^{(n)}(X)$ is divisible by $X^{q^M}-X$ for some $n\ge 1$ if and only if $A_*^{(m)}(X)$ is divisible by $(X^{q^M}-X)^{q^r}$ for some $m\ge 1$. Taking $i\to +\infty$, and using the fact that $A_*$ and $A$ commute, we see that the following statements are equivalent.
\begin{itemize}
\item There do not exist $i\ge 0$ such that $A_*^{(p^{i})}(X)$ is of the form $E(A(X))$ with $E\in \F_q[X]$ additive.
\item There do not exist $n\ge 1$ such that $A_*^{(n)}(X)$ is of the form $E(A^{(s_0)}(X))$ with $E\in \F_q[X]$ additive.
\item There do not exist $n\ge 1$ such that $A_*^{(n)}(X)$ is of the form $$E(A_*(A^{(s_0)}(X)))=E((X^{q^M}-X)^{q^r}),$$ with $E\in \F_q[X]$ additive.
\item There do not exist $n\ge 1$ such that $A_*^{(n)}(X)$ is divisible by $$(X^{q^M}-X)^{q^r};$$
\item There do not exist $n\ge 1$ such that $A_*^{(n)}(X)$ is divisible by $X^{q^M}-X$;
\item There do not exist $n\ge 1$ such that $A_*^{(n)}(y)=0$ for every $y\in \F_{q^M}$.
\item The $\F_p$-linear map $z\mapsto A_*(z)$ is not nilpotent over $\F_{q^M}$.
\end{itemize}
Now, if we assume that $z\mapsto A_*(z)$ is not nilpotent over $\F_{q^M}$, we proved that $s_i=s_0p^i$ for every $i\ge 0$. Take $n>s_0$ and let $j\ge 0$ be the unique integer such that 
$$s_0p^j=s_j< n\le s_{j+1}=s_0p^{j+1},$$
hence $s_A(n)=Mp^{j+1}$. A direct computation yields $j+1=\left\lceil\log_p \frac{n}{s_0}\right\rceil$, from where the result follows.

\subsubsection{Comments on Theorem~\ref{thm:main}}
Observe that, for real numbers $\alpha, \beta>0$, the function $f:\mathbb N\to \mathbb R_{>0}$ given by $f(n)=\alpha p^{\lceil\log_p \beta n\rceil}$ has the following properties: it is an increasing step function such that $f(\mathbb N)=\{\alpha p^j\,:\, j\ge 0\}$ and, for $$a_j:=\# f^{-1}((0, \alpha p^j])=\left\lfloor\frac{p^j}{\beta}\right\rfloor,$$ we have that 
$\lim\limits_{j\to \infty}\frac{a_{j}}{a_{j+1}}=\frac{1}{p}$. In particular, we may conclude that the graph of $f(n)$ resembles a fractal.

Although Theorem~\ref{thm:main} does not guarantee that the function $s_A(n)$ has the form $\alpha p^{\lceil\log_p \beta n\rceil}$ for generic $A$, we comment that its graph shares some of the forementioned features. Item (i) in Theorem~\ref{thm:main} readily implies that $s_A(n)$ is a step function with image set of the form $\{\alpha p^j\,:\, j\ge 0\}$. Finally, looking at the proof of item (iii) of Theorem~\ref{thm:main}, we see that $\{a_j\}_{j\ge 0}$ coincides with the sequence $\{s_j\}_{j\ge 0}$, and we proved that $s_{j+1}\ge ps_j$ for every $j\ge 0$. In particular, we conclude that $\limsup\limits_{j\to \infty}\frac{a_{j}}{a_{j+1}}\le \frac{1}{p}$. However, our results are not sufficient to obtain any conclusion on the existence of the limit $\lim\limits_{j\to \infty}\frac{a_{j+1}}{a_j}$. 

We believe that in~\cite{od}, Odoni was trying to explore a similar issue. In fact, in his results on the behaviour of the function $s_A(n)$, he proves that $s_A(n+r_n)\in \{s_A(n), ps_A(n)\}$ whenever $n$ is sufficiently large and $\{r_n\}_{n\ge 1}$ satisfies $\lim\limits_{n\to \infty}\frac{r_n}{n}=0$: see Lemma 5.2 of~\cite{od} for more details. As $p\ge 2$, item (i) in Theorem~\ref{thm:main} provides a much stronger result in this direction: we only require that $r_n\le (p-1)n$.

\section{Miscellaneous results} 
In this section we discuss some applications and improvements on Theorem~\ref{thm:main}.

\subsection{Iterating special additive polynomials}Here  we obtain an enhanced version of Theorem~\ref{thm:main} in the special case where $A$ is $q$-linearized, i.e., $A(X)$ is of the form $\sum_{i=0}^ma_iX^{q^i}$. The latter is equivalent to $A(X)=L_f(X)$ for some $f\in \F_q[X]$. For simplicity, we assume that $A$ is separable, i.e., $f(X)$ is not divisible by $X$. The general case follows easily: indeed, if $f(X)=X^rF(X)$ with $\gcd(F(X), X)=1$, then 
$$L_f^{(n)}(X)=L_{F}^{(n)}(X)^{q^{rn}}.$$
In particular, $s_A(n)=s_B(n)$ for every $n\ge 0$, where $B(X)=L_F(X)$.
For a polynomial $g\in \F_q[X]$ that is not divisible by $X$, it is direct to verify that there exists $i\ge 1$ such that $g(X)$ divides $X^i-1$. The smallest positive integer with this property is the {\em order} of $g$, which we denote by $\ord(g)$. By a simple calculation, we verify that $g(X)$ divides $X^s-1$ if and only if $s$ is divisible by $\ord(g)$. From these facts, we obtain the following explicit result that generalizes the observation given by Odoni in the case $A(X)=X^p-X$.

\begin{theorem}\label{thm:special}
Let $f\in \F_q[X]$ be a polynomial not divisible by $X$, let $f_0$ be its squarefree part and let $e$ be the least positive integer such that $f(X)$ divides $f_0(X)^e$. If $E=\ord(f_0)$, then for $A=L_f(X)$ and every $n\ge 0$, we have 
$$s_A(n)=E\cdot p^{\lceil \log_p ne\rceil }.$$
 In particular, if $c_A$ is as in Theorem~\ref{thm:main}, then $c_A=Ee$. 
\end{theorem}
\begin{proof}
By employing an Euclidean division and the properties presented in Lemma~\ref{lem:ass} we see that for nonzero $g, h\in \F_q[X]$, the polynomial $g(X)$ divides $h(X)$ if and only if $L_g(X)$ divides $L_h(X)$. In particular, since $A=L_f(X)$ is separable and $Y^{q^j}-Y=L_{X^j-1}(Y)$,  we observe that for every $s\ge 1$ the following are equivalent:
\begin{itemize}
\item $L_f^{(n)}(X)=L_{f^n}(X)$ splits completely in $\F_{q^s}$;
\item $f(X)^n$ divides $X^s-1$;
\item $s$ is divisible by $\ord(f^n)$.
\end{itemize}
Hence $s_A(n)=\ord(f^n)$. The equality $\ord(f^n)=E\cdot p^{\lceil \log_p ne\rceil }$ follows by item (ii) of Lemma 5 in~\cite{R19}. The equality $c_A=Ee$ follows by direct computation.
\end{proof}


\subsection{Proof of Theorem~\ref{thm:main2}}
From items (iii) and (iv) of Theorem~\ref{thm:main}, there exists $\beta_B>0$ such that, for every $n\ge 1$, the irreducible factors of $B^{(n)}(X)$ have all degree at most $\beta_B n$, hence $\frac{\rho_B(n)}{n}\le \beta_B$ for every $n\ge 1$. It remains to provide an absolute lower bound on $\frac{\rho_B(n)}{n}$. We observe that the formal derivative of an affine polynomial is a constant, which is nonzero if and only if the polynomial is separable. The latter is also equivalent to the affine polynomial not being a perfect $p$-th power. Write $B(X)=B_0(X)^{p^m}$, where $m\ge 0$ and $B_0(X)$ is not of the form $S(X)^p$. From hypothesis, $\deg(B_0)=p^d$ for some $d\ge 1$. From the initial observation, it follows by induction on $n\ge 1$ that $B^{(n)}(X)=\tilde{B}_n(X)^{p^{mn}}$ for some affine separable polynomial $\tilde{B}_n$ of degree $p^{dn}$. 
In particular, the set $\mathcal Z_n$ of the distinct roots of $B^{(n)}(X)$ has cardinality $p^{dn}$ and so it follows by the definition that  \begin{equation}\label{eq:av}\rho_B(n)=\frac{p^{dn}}{N_B(n)}.\end{equation}
Our aim is to provide a non trivial bound for $N_B(n)$. Let $\F_{q^M}$ be the splitting field of $B(X)$ over $\F_q$ and let $n$ be a positive integer which we will assume to be sufficiently large. From items (i) and (ii) in Theorem~\ref{thm:main}, we have that $s_B(n)=Mp^t$ for some $t=t(n)$ such that $0\le t<\log_p n+2$. Hence $B^{(n)}(X)$ splits completely in $\F_{q^{Mp^t}}$, i.e., $\mathcal Z_n\subseteq \F_{q^{Mp^t}}$. Moreover, from items (iii) and (iv) in Theorem~\ref{thm:main}, we have that $t\to +\infty$ as $n\to +\infty$, so we can assume that $t$ is sufficiently large.

As $B$ is affine, for each $0\le j\le t$, the set $\mathcal U_j:=\mathcal Z_n\cap \F_{q^{Mp^j}}$ is an $\F_p$-affine space. For each $0\le j\le t$, let $\delta_j$ be the dimension of $\mathcal U_j$. From the inclusions $\F_{q^M}\subset \F_{q^{Mp}}\subset\cdots\subset \F_{q^{Mp^t}}$ we obtain  
$\delta_0\le \cdots\le \delta_{t}=dn$.
Set $\mathcal Y_0=\mathcal U_0$ and, for each $1\le j\le t$, set $\mathcal Y_j=\mathcal U_j\setminus \mathcal U_{j-1}$. Therefore, $\# \mathcal Y_j=p^{\delta_j}-p^{\delta_{j-1}}$ for every $1\le j\le t$. Moreover, if $1\le j\le t$ and $\alpha\in \mathcal Y_j$, then $\alpha\in \F_{q^{Mp^j}}\setminus \F_{q^{Mp^{j-1}}}$ and so the degree $D_{\alpha}$ of the minimal polynomial of $\alpha$ over $\F_{q}$ divides $Mp^j$ but it does not divide $Mp^{j-1}$. The latter implies that $D_{\alpha}$ is divisible by $p^j$ and so $D_{\alpha}\ge p^j$. 
Therefore, for every $1\le j\le t$, the set $\mathcal Y_j$ produces at most $\frac{p^{\delta_j}-p^{\delta_{j-1}}}{p^j}$ distinct irreducible factors of $B^{(n)}(X)$. As $\mathcal Z_n=\bigcup_{i=0}^t\mathcal Y_j$, we obtain
\begin{equation}\label{eq:ineq}N_B(n)\le p^{\delta_0}+\sum_{j=1}^t\frac{p^{\delta_j}-p^{\delta_{j-1}}}{p^j}\le \sum_{j=0}^{t}p^{\delta_j-j}.\end{equation}
We now provide nontrivial bounds for $\delta_j$.  Set $$\mathcal W_j=\{u-v\,:\, u, v\in \mathcal U_j\}.$$ 
Observe that $B^{(n)}(X)=A^{(n)}(X)+\beta_n$ for some $\beta_n\in \F_q$. Therefore, $\mathcal W_j\subseteq \F_{q^{Mp^j}}$ is an $\F_p$-vector space of dimension $\delta_j$, whose elements are roots of $A^{(n)}(X)$. Moreover, the roots of $A^{(n)}(X)$ comprise the set $\mathcal S=\cup_{i=0}^t\mathcal W_j$. As $A$ is additive, it follows that $A(0)=0$ and so $\mathcal S$ also contains the roots of $A^{(r)}(X)$ for $1\le r\le n$.  As $pr\ge r+1$, from item (i) in Theorem~\ref{thm:main}, we obtain that $s_A(r+1)\in \{s_A(r), ps_A(r)\}$ for every $r\ge 1$. In other words, the degree of the splitting field of consecutive iterates of $A$ increases by a factor $p$ each time it increases. The latter implies that we must have the strict inclusions $\mathcal W_0\subsetneq\cdots \subsetneq \mathcal W_t$ and so we obtain $\delta_0<\cdots<\delta_t$. Therefore, $\delta_j\le \delta_t-(t-j)=dn+j-t$ and so $\delta_j-j\le dn-t$ for every $0\le j\le t$. Since $\mathcal U_j\subseteq \F_{q^{Mp^j}}$, if $q=p^R$, we also have the trivial bound $\# \mathcal U_j\le \# \F_{q^{Mp^j}}$, i.e.,  $\delta_j-j\le \delta_j\le RMp^j$. 
Let $k=k(t)$ be the smallest positive integer such that $RMp^{t-k}\le dn-t$, i.e., 
$$p^{-k}\le \frac{dn-t}{RMp^t}=\frac{dn-t}{Rs_B(n)}=\frac{d}{R}\cdot \frac{n}{s_B(n)}\left(1-\frac{t}{dn}\right).$$
From items (iii) and (iv) in Theorem~\ref{thm:main}, the number $\frac{n}{s_B(n)}$ lies in an interval $[a, b]$ with $0<a<b<+\infty$. Moreover, as $0\le t<\log_p n+2$ we obtain $\lim\limits_{n\to \infty}\frac{t}{n}=0$ and
so we conclude that there exists $c>0$ such that $k(t)\le c$ whenever $n$ (hence $t$) is sufficiently large. In particular, taking $n$ large enough, we have $k(t)< t$ and so from Ineq.~\eqref{eq:ineq} we obtain
\begin{align*}N_B(n)\le \sum_{j=0}^{t-k(t)}p^{RMp^j}+\sum_{j=t-k(t)+1}^{t}p^{dn-t}& \le 2p^{RMp^{t-k(t)}}+k(t)p^{dn-t}\\ {} & \le (k(t)+2)p^{dn-t}.\end{align*}
The latter combined with Eq.~\eqref{eq:av} yields $$\rho_B(n)\ge \frac{p^t}{k(t)+2}=\frac{s_B(n)}{M(k(t)+2)}\ge \frac{s_B(n)}{M(c+2)}=n\cdot \frac{s_B(n)}{n}\cdot \frac{1}{M(c+2)},$$ and so the result follows by items (iii) and (iv) in Theorem~\ref{thm:main}.

\subsection{Dynamics of linear maps over finite fields}
Here we provide statistics on the periodic points under the iteration of linear maps over finite fields.  We start with a general structural theorem.
\begin{theorem}\label{thm:dyn}
Let $A\in \F_q[X]$ be an additive polynomial and let $n\ge 1$. Then we have a direct sum of $\F_p$-vector spaces $\F_{q^n}=\mathcal W_0\oplus \mathcal W_1$, where 
$\mathcal W_0=\bigcup_{t\ge 0}\{y\in \F_{q^n}\,:\, A^{(t)}(y)=0\}$ and $\mathcal W_1$ comprises the elements of $\F_{q^n}$ that are periodic under $A$. In particular, the number $\pi_A(n)$ of periodic points of $\F_{q^n}$ under the map $z\mapsto A(z)$ satisfies
$$\pi_A(n)=\# \mathcal W_1=\frac{q^n}{\#\mathcal W_0}.$$
\end{theorem}
\begin{proof}
Observe that, as $A$ is additive, the map $z\mapsto A(z)$ is an $\F_p$-linear map over $\F_{q^n}$. Let $P(X)$ be its characteristic polynomial and write $P(X)=X^rQ(X)$, where $r\ge 0$ and $\gcd(Q(X), X)=1$. We now employ a classical Linear Algebra result: we have $\F_{q^n}= V_0\oplus V_1$, where $V_0$ (resp. $V_1$) is the 
$A$-invariant subspace of $\F_{q^n}$ associated to the factor $X^r$ (resp. $Q(X)$). Moreover, $V_1$ is the largest $A$-invariant subspace where $A$ restricts to a bijection. We clearly have $V_0=\mathcal W_0$. As $\F_{q^n}$ is finite, $A$ restricts to a bijection on an $A$-invariant subspace if and only if the elements of this subspace are periodic under the map $z\mapsto A(z)$. We directly verify that $\mathcal W_1$ is an $A$-invariant subspace, hence $V_1=\mathcal W_1$ and the result follows.
\end{proof}

\subsubsection{Proof of Theorem~\ref{thm:main3}}

The case where $A(X)=aX^{p^h}$ follows by direct computations since $z\mapsto A(z)$ is the zero map for $a=0$ and, if $a\ne 0$, such map is a permutation of the whole field $\overline{\F}_q$. We now assume that $A(X)$ is not of this form. As in the proof of Theorem~\ref{thm:main2}, for every integer $n\ge 1$, we verify that there exists integers $r\ge 0$ and $d\ge 1$ such that $A^{(n)}(X)=\tilde{A}_n(X)^{p^{nr}}$ for some additive separable polynomial $\tilde{A}_n(X)\in\F_q[X]$ of degree $p^{dn}$.

For each positive integer $n$, set $\Delta(n)=\#\bigcup_{t\ge 0}\{y\in \F_{q^n}\,:\, A^{(t)}(y)=0\}$. Theorem~\ref{thm:dyn} implies the identity $\frac{\pi_A(n)}{q^n}=\Delta(n)^{-1}$ so we only need to find values of $n$ where $\Delta(n)$ is relatively small/sufficiently large. Let $\F_{q^M}$ be the splitting field of $A(X)$ over $\F_q$ and let $N$ be the largest positive integer such that $\F_{q^M}$ contains a root of $A^{(N)}(X)$ that is not a root of $A^{(N-1)}(X)$ (under the convention that $A^{(0)}(X)=X$). From item (i) in Theorem~\ref{thm:main}, the roots of $A^{(j)}(X)$ that are not in $\F_{q^M}$ must be contained in $\F_{q^{Mp^s}}$ for some $s\ge 1$. In particular, we obtain $\Delta(Mt)\le p^{dN}$ whenever $\gcd(t, p)=1$. Therefore, 
$$\limsup\limits_{n\to \infty}\frac{\pi_A(n)}{n}\ge\limsup\limits_{t\to \infty}\frac{\pi_A(Mt)}{q^{Mt}}\ge p^{-dN}>0.$$
On the other hand, from item (i) in Theorem~\ref{thm:main}, we have that $A^{(p^i)}(X)$ splits completely in $\F_{q^{Mp^i}}$ and so $\Delta(Mp^i)\ge p^{dp^i}$ for every integer $i\ge 0$. Therefore,
$$\liminf\limits_{n\to \infty}\frac{\pi_A(n)}{n}\le \liminf\limits_{i\to \infty}\frac{\pi(Mp^i)}{q^{Mp^i}}\le \liminf\limits_{i\to\infty} p^{-dp^i}=0.$$

\end{document}